\documentclass[]{elsarticle}
\usepackage{amsmath}
\usepackage{amsfonts}
\usepackage{amssymb}
\usepackage[utf8]{inputenc}
\usepackage[T1]{fontenc}
\usepackage[colorlinks=true]{hyperref}
\usepackage[dvipsnames]{xcolor}
\usepackage{graphicx}
\usepackage[normalem]{ulem}
\usepackage{enumerate}
\usepackage{soul}

\newcounter{defi}
\newtheorem{definition}[defi]{Definition}

\newcounter{corr}
\newtheorem{corollary}[corr]{Corollary}

\newtheorem{theorem}{Theorem}

\newenvironment{proof}[1][Proof]{\noindent\textbf{#1.} }{\ \rule{0.5em}{0.5em}}

\newcommand{\D}{\text{d}}

\begin{document}

\begin{frontmatter}

\title{Identifying Berwald Finsler Geometries}

\author{Christian Pfeifer}
\ead{christian.pfeifer@ut.ee}
\ead{christian.pfeifer@zarm.uni-bremen.de}
\address{Laboratory of Theoretical Physics, Institute of Physics, University of Tartu, W. Ostwaldi 1, 50411 Tartu, Estonia}
\address{ZARM, University of Bremen, 28359 Bremen, Germany.}

\author{Sjors Heefer\corref{mycorrespondingauthor}}
\ead{s.j.heefer@tue.nl}
\author{Andrea Fuster}
\ead{a.fuster@tue.nl}
\address{Department of Mathematics and Computer Science,\\Eindhoven University of Technology, Eindhoven, The Netherlands}

\cortext[mycorrespondingauthor]{Corresponding author}

\begin{abstract}
Berwald geometries are Finsler geometries close to (pseudo)-Riemannian geometries. We establish a simple first order partial differential equation as necessary and sufficient condition, which a given Finsler Lagrangian has to satisfy to be of Berwald type. Applied to $(\alpha,\beta)$-Finsler spaces or spacetimes, respectively, this reduces to a necessary and sufficient condition for the Levi-Civita covariant derivative of the geometry defining $1$-form. We illustrate our results with novel examples of $(\alpha,\beta)$-Berwald geometries which represent Finslerian versions of Kundt (constant scalar invariant) spacetimes. The results generalize earlier findings by Tavakol and van den Bergh, as well as the Berwald conditions for Randers and m-Kropina resp. very special/general relativity geometries.
\end{abstract}


\end{frontmatter}

\maketitle

%\tableofcontents

%%%%%%%%%%%%%%%%%%%%%%%%%%%%%%%%%%%%%%%%%%%%%%%%%%%%%%%%%%%%%%%%%%%%%%%%%%%%%%%%%%%%%%%
\section{Introduction}\label{sec:intro}
In Finsler geometry the geometry of a manifold is derived from a $1$-homogeneous length measure for curves and its corresponding canonical Cartan non-linear connection on the tangent bundle, instead of from a metric and its Levi-Civita connection, as it is done in (pseudo)-Riemannian geometry. The first one who considered the possibility to derive the geometry of spacetime from a more general length measure than the a metric one, was Riemann himself \cite{Riemann}. Only years later Finsler started a systematic study of such manifolds \cite{Finsler}. Since then, Finsler geometry became an established field in mathematics, usually as generalization of Riemannian geometry, \cite{Bao,Miron}, and gained interest in its application to physics as generalization of pseudo-Riemannian (Lorentzian) geometry \cite{Javaloyes:2018lex,Pfeifer:2019wus,Hohmann:2018rpp,Hohmann:2019sni,Bernal:2020bul}.

Berwald geometries are Finslerian geometries whose canonical Cartan non-linear connection on the tangent bundle is in one to one correspondence to an affine connection on the base manifold \cite{Berwald1926}. In other words the Cartan non-linear connection is linear in its dependence on tangent directions. Berwald geometries can be regarded as Finsler geometries close to (pseudo)-Riemannian geometries. 

In this article we present a necessary and sufficient condition, a first order partial differential equation, which classifies if a Finsler geometry is of Berwald type or not. After introducing the mathematical setup in Section \ref{sec:fins}, the general theorem will be presented in Section \ref{sec:TheTheorem}. We will apply it to $(\alpha,\beta)$-Finsler spaces and spacetimes in Section \ref{sec:AB}, where we find a necessary and sufficient condition on the Levi-Civita covariant derivative of the geometry defining $1$-form, which ensures that the geometry is Berwald. This condition then is used to find Berwald Finsler Lagrangians for which the $1$-form does not have to be covariantly constant. We illustrate our results with novel examples of $(\alpha,\beta)$-Berwald geometries, which represent Finslerian versions of Kundt (constant scalar invariant) spacetimes for both covariantly constant and not covariantly constant $1$-forms. In the end we conclude in Section \ref{sec:conc}.

The results presented here generalize the findings of Tavakol and Van Den Bergh \cite{TAVAKOL198523,Tavakol1986,Tavakol2009} as well as conditions that where found in the context of $(\alpha,\beta)$-Finsler spaces \cite{Szilasi2011,Matsumoto,ShenBacsoCheng} and spacetimes~\cite{Fuster:2018djw,Kostelecky:2011qz,Gomez-Lobo:2016qik}.

%%%%%%%%%%%%%%%%%%%%%%%%%%%%%%%%%%%%%%%%%%%%%%%%%%%%%%%%%%%%%%%%%%%%%%%%%%%%%%%%%%%%%%%
\section{Finsler geometry}\label{sec:fins}
We consider an $n$-dimensional smooth manifold $M$ and its tangent bundle $TM$. The later is equipped with manifold induced coordinates, which means that an element $X \in TM$ is labeled as $ X= \dot x^a \partial_a|_x = (x,\dot x)$, where $x$ denotes local coordinates on $M$ and $\dot x$ the coordinate basis components of the vector $X\in T_xM$. The canonical local coordinate basis of the tangent spaces to the tangent bundle, $T_{(x,\dot x)}TM$, is labeled by $\{\partial_a = \frac{\partial}{\partial x^a}, \dot{\partial}_a = \frac{\partial}{\partial \dot x^a} \}$.

A Finsler geometry is a smooth manifold $M$ endowed with a continuous Finsler Lagrangian $L: TM \rightarrow \mathbb{R}$, which has at least the following properties:
\begin{itemize}
	\item $L(x,\dot x)$ is positively homogeneous of degree two with respect to $\dot x$, i.e.\ $L(x, \lambda \dot x) = \lambda^2 L(x,\dot x),\quad \forall \lambda>0$;
	\item the $L$-metric, the Hessian of $L$ with respect to the tangent space coordinates,
	\begin{align}
		g^L_{ab}(x,\dot x) = \frac{1}{2}\dot{\partial}_a \dot{\partial}_b L(x,\dot x)\,,
	\end{align}
	is non-degenerate and $L$ is smooth on a conic subbundle $\mathcal{A}$ of $TM$, such that $TM \setminus \mathcal{A}$ is of measure zero.
\end{itemize}
The Finsler function $F$, which defines the length measures for curves on $M$,
\begin{align}
	S[x] = \int d\tau\ F(x,\dot x)\,,
\end{align}
is derived from $L$ as $F(x,\dot x)=\sqrt{|L(x,\dot x)|}$. We denote Finslerian geometries by a tuple $(M,L)$.

We do not specify the signature of the $L$-metric nor do we further specify the conic subbundle $\mathcal{A}$\footnote{For Finsler spaces $\mathcal{A}$ is usually $TM\setminus\{0\}$, while $\mathcal{A}$ has a more complicated structure for Finsler spacetimes \cite{Javaloyes:2018lex,Hohmann:2018rpp,Hohmann:2019sni,Bernal:2020bul}.}, since the result we present here holds for Finsler spaces with positive definite $L$-metric, for Finsler spacetimes with $L$-metric of Lorentzian signature and for any other choice of signature of the $L$-metric one may like to impose. The only necessary requirement is that the $L$-metric be invertible on a large enough subset of the tangent bundle. When we do not need to distinguish between signatures, we use the term \textit{Finslerian geometry}.

The fundamental object that describes the geometry derived from $L$ is the Cartan non-linear connection, which is the unique torsion-free, $L$-compatible, homogeneous non-linear connection on $\mathcal A$. Its connection coefficients are given by
\begin{align}\label{eq:ncon}
	N^a{}_b(x,\dot x) = \frac{1}{4}\dot{\partial}_b(g^{Laq}(x,\dot x)(\dot x^m \partial_m \dot{\partial}_q L(x,\dot x) - \partial_q L(x,\dot x)))\,.
\end{align}
It defines the Finsler geodesic equation, for arc length parametrized curves $x(\tau)$ on $M$, 
\begin{align}
	\ddot x^a + 2 G^a(x,\dot x) = 0
\end{align}
in terms of the geodesic spray 
\begin{align}\label{eq:geodspray}
	2 G^a(x,\dot x) = N^a{}_b(x,\dot x) \dot x^b\,,
\end{align}
as well as the horizontal derivative $\delta_a = \partial_a - N^b{}_a \dot{\partial}_b$. The $L$-compatibility as defining property manifests in the fact that the horizontal derivative annihilates $L$ and that the $L$-metric is covariantly constant
\begin{align}\label{eq:dL=0}
	\delta_a L(x,\dot x) = 0\,,\quad \dot x^c\delta_c g^L_{ab}(x,\dot x) - N^c{}_a(x,\dot x) g^L_{cb}(x,\dot x) - N^c{}_b(x,\dot x) g^L_{ca}(x,\dot x) = 0\,.
\end{align}
This property of Finsler geometry is crucial for the proof of the main theorem of this article.

For general Finslerian geometries, the Cartan non-linear connection coefficients \eqref{eq:ncon} are $1$-homogeneous with respect to $\dot x$.

\begin{definition}
	A Finsler geometry $(M,L)$ is said to be of Berwald type if the Cartan non-linear connection coefficients are linear in $\dot x$
	\begin{align}\label{eq:NLConn}
		N^a{}_{b}(x,\dot x) = \Xi^{a}{}_{bc}(x)\dot x^c\,.
	\end{align}
\end{definition}
This implies that the functions $\Xi^{a}{}_{bc}(x)$ are connection coefficients of an affine connection on $M$. Equivalently one can define Berwald geometries by the demand that the geodesic spray \eqref{eq:geodspray} be quadratic in $\dot x$. The definition goes back to the original article by Berwald \cite{Berwald1926}, a modern overview on Berwald geometries can be found in \cite{Szilasi2011}. For positive definite Finsler geometry a well-known result by Szabo \cite{Szabo} states that Berwald spaces are Riemann-metrizable, i.e.\ the connection coefficients $\Xi^a{}_{bc}$ are the Levi-Civita connection coefficients of a Riemannian metric. For indefinite Finsler geometries, however, this result does not hold in general, in the sense that indefinite Finsler geometries are not necessarily metrizable in the pseudo-Riemannian sense ~\cite{Fuster:2020upk}.

The first condition that comes to mind to characterise Berwald geometries is $\dot{\partial}_b\dot{\partial}_c \dot{\partial}_d G^a(x,\dot x) = \dot{\partial}_b\dot{\partial}_c N^a{}_d(x,\dot x) = 0$, which in general is difficult to evaluate and analyse. In the next section we present a simpler, first order partial differential equation, which is necessary and sufficient for a Finsler Lagrangian to be of Berwald type.

%%%%%%%%%%%%%%%%%%%%%%%%%%%%%%%%%%%%%%%%%%%%%%%%%%%%%%%%%%%%%%%%%%%%%%%%%%%%%%%%%%%%%%%
\section{The Berwald Condition}\label{sec:TheTheorem}

The strategy to find the \emph{Berwald condition} in this section is as follows. We introduce an auxiliary positive definite metric $g$ on $M$, which defines the non-vanishing function $A=g(\dot x, \dot x)$, in local coordinates define by components of the metric as $A= g_{ab}(x)\dot x^a \dot x^b$ on $TM\setminus \{0\}$. With help of this function we can introduce the factor $\Omega = \tfrac{L}{A}$ on $\mathcal{A}\subseteq TM\setminus \{0\}$ to rewrite any Finsler Lagrangian as $L=A\Omega$. For Finsler Lagrangians of this form it is possible to derive a necessary and sufficient  first order partial differential equation $\Omega$ has to satisfy as in order for the Finslerian geometry to be of Berwald type.

%\SH{[It seems to me that we cannot require $g$ to be positive definite here, because all applications in the remainder of the article assume that $\alpha$ may have any signature. We should probably allow $g$ to be pseudo-Riemannian, but consider only the subset of $\mathcal A$ where $A= g(\dot x,\dot x)\neq 0$.]}

Having proven this statement we will discuss the dependence of this result on the choice of the metric $g$. We will demonstrate that if $L$ satisfies the Berwald condition for a decomposition with respect to the metric $g$, then it satisfies the Berwald condition for a decomposition $L=\tilde g \tilde \Omega$ in terms of any other metric $\tilde g$ on the set $\mathcal{B}\subset TM\setminus \{0\}$, where $\mathcal{B}$ is the set on which $\tilde g(\dot x, \dot x)\neq 0$. The other way around, this implies that if one would start with a decomposition of $L$ with respect to a metric $\tilde g$ for which $\mathcal{B}\neq TM\setminus \{0\}$ and one would prove the Berwald condition only on $\mathcal{B}$, then one can extend the proof by changing the decomposition of $L$ from $\tilde g$ to any positive definite metric $g$.

\begin{theorem}\label{thm:1}
	Let $(M,L)$ be a manifold equipped with a Finsler Lagrangian and let $g$ be an auxiliary positive definite metric on $M$. On $\mathcal{A}$ we can write
	\begin{align}
		L = \Omega A\,,
	\end{align}
	with $\Omega = L A^{-1}$. 
	\begin{enumerate}
		\item On the subbundle $\mathcal{A}$, $(M,L)$ is of Berwald type if and only if there exists  a $(1,2)$-tensor field $T$ on $M$, such that $\Omega$ satisfies
			\begin{align}\label{eq:BerwCond}
				\partial_a \Omega - \Gamma^b{}_{ac}\dot x^c\dot{\partial}_b \Omega = T^b{}_{ac}\dot x^c \left(\dot{\partial}_b \Omega + \frac{2 \dot x_b \Omega}{A}\right)\,,
			\end{align}
			where $\Gamma^a{}_{bc} = \tfrac{1}{2}g^{am}(\partial_bg_{ac}+\partial_cg_{ab} - \partial_mg_{bc} )$ are the Christoffel symbols of the auxillary metric g.
		\item The geodesic spray components of the Berwald type Finslerian geometry are given by
		\begin{align}\label{eq:geodsprayB}
			G^a = \frac{1}{2}(\Gamma^a{}_{bc} + T^a{}_{bc})\dot x^b \dot x^c\,.
		\end{align}
	\end{enumerate}
\end{theorem}

\noindent The proof of the theorem uses ideas from Tavakol and van den Bergh \cite{Tavakol1986}.\vspace{10pt}

\begin{proof}[Proof of Theorem 1]
	The starting point for the proof is to realize that $\delta_a L = 0$ implies the equation
	\begin{align}\label{eq:deltaLproof}
		 0&=A \delta_a \Omega + 2 \dot x_c (\Gamma^c{}_{ab}\dot x^b - N^c{}_a)\Omega\nonumber\\
		 &= A (\partial_a\Omega - N^b{}_a\dot{\partial}_b\Omega) + 2 \dot x_c (\Gamma^c{}_{ab}\dot x^b - N^c{}_a)\Omega\,.
	\end{align}
	
	Assume $L$ defines a Finslerian geometry of Berwald type. Then its connection coefficients can be written as $N^a{}_b(x,\dot x) = \Xi^a{}_{bc}(x)\dot x^c$, see \eqref{eq:NLConn}. Since $\Xi^a{}_{bc}(x)$ are affine connection coefficients, the can written as a sum of the Christoffel symbols of the auxiliary metric $g$ and a tensor field $T^a{}_{bc}(x)$, which yields $N^a{}_b = (\Gamma^a{}_{bc}(x) + T^a{}_{bc}(x))\dot x^c$. Using this in equation \eqref{eq:deltaLproof} we obtain the Berwald condition \eqref{eq:BerwCond} with tensor $T^a{}_{bc} = \Xi^a{}_{bc}(x) - \Gamma^a{}_{bc}(x)$.
	
	The other way around, assume \eqref{eq:BerwCond} holds for some tensor components $T^a{}_{bc}$. Solving for $\partial_a \Omega$ gives
	\begin{align}
		\partial_a \Omega = \Gamma^b{}_{ac}\dot x^c\dot{\partial}_b \Omega + T^b{}_{ac}\dot x^c \left(\dot{\partial}_b \Omega + \frac{2 \dot x_b \Omega}{A}\right)\,.
	\end{align} 
	Employing this identity in \eqref{eq:deltaLproof} yields
	\begin{align}\label{eq:proof2}
		((\Gamma^b{}_{ac} + T^b{}_{ac})\dot x^c - N^b{}_a) (A\dot{\partial}_b \Omega + 2 \dot x_b \Omega) = 0\,.
	\end{align}
	Realizing that $\dot{\partial}_b L = A \dot{\partial}_b \Omega + 2 \dot x_b \Omega$ and applying another $\dot \partial_d$ derivative to the above equation results in
	\begin{align} \label{eq:proof3}
		2 g^L_{bd}((\Gamma^b{}_{ac} + T^b{}_{ac})\dot x^c - N^b{}_a) + (A\dot{\partial}_b \Omega + 2 \dot x_b \Omega)(\Gamma^b{}_{ad} + T^b{}_{ad} - \dot{\partial}_dN^b{}_a) = 0\,.
	\end{align}
	Multiplication with $\dot x^a$ simplifies the expression to 
	\begin{align}\label{eq:proof4}
		g^L_{bd}\dot x^a((\Gamma^b{}_{ac} + T^b{}_{ac})\dot x^c - N^b{}_a)=0\,,
	\end{align}
	since the second term in \eqref{eq:proof3} becomes precisely the left hand side of \eqref{eq:proof2}. This statement is true due to the symmetry in the lower indices of $\dot{\partial}_dN^b{}_a$ (i.e. torsion-freeness), the $1$-homogeneity of $N^a{}_b$ with respect to $\dot x^a$ and Euler's theorem for homogeneous functions\footnote{Let $f(x)$ be a $r$-homogeneous function with respect to x, i.e.\ $f(\lambda x) = \lambda^r f(x)$, then $x^a\partial_a f(x) = r f(x)$, see for example \cite{Bao}}.
	
	Since we assumed the $L$-metric to be non-degenerate on $\mathcal{A}$, equation \eqref{eq:proof4} yields the desired result that $ N^b{}_a \dot x^a = (\Gamma^b{}_{ac} + T^b{}_{ac})\dot x^c\dot x^a = 2 G^a$, which completes the proof.
\end{proof}

This theorem offers a simple way how to find Berwald Finsler geometries in practice, since it reduces the problem to finding a function $\Omega$ which satisfies a first order partial differential equation. Moreover, it is a direct extension of the results of Tavakol and van den Berg \cite{Tavakol1986}, which is an immediate corollary.

\begin{corollary}
	A Finslerian geometry 
	\begin{align}
		L = e^{2 \sigma} A
	\end{align}
	has a geodesic spray given by  $G^a = \frac{1}{2} \Gamma^a{}_{bc} \dot x^b \dot x^c$ if and only if 
	\begin{align}\label{eq:TvdB}
	\partial_a \sigma - \Gamma^b{}_{ac}\dot x^c\dot{\partial}_b \sigma = 0\,.
	\end{align}
\end{corollary}

\begin{proof}[Proof of Corollary 1]
	The proof is obtained by simply inserting $\Omega = e^{2 \sigma}$ and $T^a{}_{bc}= 0$ into \eqref{eq:BerwCond}.
\end{proof}

To conclude this section we discuss the decomposition of the Finsler Lagrangian with respect to another auxiliary metric $\tilde g$ of arbitrary signature.

Since $\tilde g$ is of arbitrary signature the function $\tilde A = \tilde g_{ab}(x)\dot x^a \dot x^b$ is non-vanishing only on $\mathcal{B}\subset TM\setminus \{0\}$. Thus, on $\mathcal{B}$ we can write
\begin{align}\label{eq:differentmetrics}
	L = A \Omega = \tilde A \tilde \Omega.
\end{align}
If $L$ is of Berwald type then $A$ and $\Omega$ satisfy \eqref{eq:BerwCond}, which implies the existence of the tensor field with components $T^a{}_{bc}$. Moreover $\delta_a L=0$ implies for $\tilde A$ and $\tilde \Omega$
\begin{align}\label{eq:deltatilde}
		\tilde A (\partial_a \tilde \Omega - N^b{}_a\dot{\partial}_b\tilde \Omega) + 2 \dot {\tilde x}_c (\tilde \Gamma^c{}_{ab}\dot x^b - N^c{}_a)\tilde \Omega = 0\,.
\end{align}
Using again that we assumed $L$ is Berwald, we can expand $N^a{}_b = (\Gamma^a{}_{bc}(x) + T^a{}_{bc}(x))\dot x^c$ and define a new tensor field with components $\tilde T^a{}_{bc} = T^a{}_{bc} + \Gamma^a{}_{bc} - \tilde \Gamma^a{}_{bc}$. Using these in \eqref{eq:deltatilde} yields that on $\mathcal{B}$, $\tilde A$ and $\tilde \Omega$ satisfy
\begin{align}
	\partial_a \tilde \Omega - \Gamma^b{}_{ac}\dot x^c\dot{\partial}_b \tilde \Omega = \tilde T^b{}_{ac}\dot x^c \left(\dot{\partial}_b \tilde \Omega + \frac{2 \dot x_b \Omega}{\tilde A}\right)\,,
\end{align}
which is again the Berwald condition \eqref{eq:BerwCond} just expressed in terms of $\tilde A$ and $\tilde \Omega$. 

Alternatively one could start by using the Berwald condition on $\mathcal{B}$ for $\tilde g$ and $\tilde \Omega$ to determine if a Finslerian geometry is of Berwald type or not, and then rewrite the condition in terms of a positive definite metric $g$ to extend the result to  the whole set $\mathcal{A}$.\footnote{While this article was under revision, this has also been pointed out in \cite[Sec. 4.1]{Hohmann:2020mgs}.}

%%%%%%%%%%%%%%%%%%%%%%%%%%%%%%%%%%%%%%%%%%%%%%%%%%%%%%%%%%%%%%%%%%%%%%%%%%%%%%%%%%%%%%%
\section{$(\alpha,\beta)$-Berwald Finsler geometries}\label{sec:AB}
We will now apply Theorem \ref{thm:1} to a special class of Finslerian geometries, so-called $(\alpha,\beta)$-geometries, and demonstrate the application of our findings on some examples. An $(\alpha,\beta)$-geometry is one for which the Finsler Lagrangian $L = L(\alpha,\beta)$ is constructed from a Riemannian metric $a$ and a $1$-form, $b = b_a(x)dx^a$, on $M$, defining the variables $\alpha = \alpha(\dot x,\dot x) = \sqrt{a_{ab}(x)\dot x^a \dot x^b}$, and $\beta = \beta(\dot x) = b_a(x)\dot x^a$. Since we want to include the possibility of $a$ being pseudo-Riemannian, we parametrize such Finsler geometries by the variables $\mathfrak{a}=\mathfrak{a}(\alpha) = \alpha^2 =a_{ab}(x)\dot x^a \dot x^b$ and $\beta$.

This kind of Finslerian geometries, constructed from simple building blocks, have been considered in mathematics and physics. Famous examples are for example Randers geometry \cite{Randers} and the m-Kropina \cite{Kropina} respectively very special/general relativity (VGR) geometries \cite{Cohen:2006ky,Gibbons:2007iu,Fuster:2018djw}. The latter were originally Bogoslovsky in the context of invariance properties of the massless wave equation \cite{Bogoslovsky}.

For Randers geometries, which are defined by the Finsler Lagrangian $L = (\sqrt{\mathfrak{a}}+\beta)^2$, it is well known that they are Berwald if and only if the $1$-form $b$ is covariantly constant with respect to the Levi-Civita connection of the metric~$a$, $\nabla_a b_b = 0$. For m-Kropina/$VGR$ geometries of the type $L= \mathfrak{a}^n \beta^{2(1-n)}$ it is known that there must exist a function $C(x)$ on $M$ such that $\nabla_a b_b = C(x)(2(1+n)b_a b_b - n a^{-1}(b,b)a_{ab})$ holds, for $L$ to be of Berwald type, see~\cite{Fuster:2018djw}.

Using Theorem \ref{thm:1}, we find a necessary and sufficient condition for general $(\alpha,\beta)$-Finsler geometries to be of Berwald type, characterized by the Levi-Civita covariant derivative of the involved $1$-form. As the choice of the auxiliary metric in Theorem \ref{thm:1} is completely arbitrary, we choose here to use the metric $a$ itself, i.e. $A = \mathfrak a = \alpha^2$.

\begin{corollary}\label{corr:AB}
	Let $\mathfrak{a} =  a(\dot x, \dot x) = a_{ab}(x)\dot x^a \dot x^b$ and $\beta=\beta(\dot x) = b_a(x) \dot x^a$, where $b_a$ are the components of a $1$-form $b$ on~$M$. Any $(\alpha,\beta)$-Finsler geometry $(M,L)$, i.e. $L = L(\mathfrak{a},\beta) = \Omega(\frac{\beta^2}{\mathfrak{a}})\mathfrak{a}$ is of Berwald type if and only if there exists a tensor field $T^a{}_{bc}$ on $M$ such that
	\begin{align}\label{eq:CDBerwaldAB2}
		\dot x^c \nabla_a b_c = T^b{}_{ac}\dot x^c b_b + T^b{}_{ac}\dot x^c \dot x_b \left(\frac{\Omega}{\Omega' \beta} - \frac{\beta}{\mathfrak{a}}\right)\,,
	\end{align}
	where $\Omega'$ denotes the derivative of $\Omega$ with respect to its argument. A further $\dot{\partial}_d$ derivative yields the equivalent condition
	\begin{align}\label{eq:CDBerwaldAB}
	\nabla_a b_d = T^b{}_{ad}b_b &+ \left(T^b{}_{ad}\dot x_b + T^b{}_{ac}\dot x^c g_{db}\right)\left(\frac{\Omega}{\Omega' \beta} - \frac{\beta}{\mathfrak{a}}\right) \nonumber\\
	&+ T^b{}_{ac}\dot x^c \dot x_b \left[ \dot x_d \frac{2 \beta}{\mathfrak{a}^2} \frac{\Omega \Omega''}{\Omega'^2} -  b_d \left( -\frac{1}{\mathfrak{a}} + \frac{\Omega}{\Omega' \beta^2} + \frac{2 \Omega \Omega''}{\Omega'^2 \mathfrak{a}} \right) \right]\,.
	\end{align}
\end{corollary}
	
\begin{proof}[Proof of Corollary 2]
	Consider $\Omega = \Omega(\frac{\beta^2}{\mathfrak{a}})$. Employing the identities
	\begin{align}
		\partial_a \Omega &= \Omega'\frac{2 \beta}{\mathfrak{a}}\left( \dot x^c \partial_ab_c - \frac{\beta}{\mathfrak{a}}\Gamma^b{}_{ac}\dot x^c \dot x_b \right)\,,\\
		\dot {\partial}_a \Omega &= \Omega'\frac{2 \beta}{\mathfrak{a}}\left(b_a - \frac{\beta}{\mathfrak{a}}\dot x_a\right)\,,
	\end{align} 
	in \eqref{eq:BerwCond} yields the desired expression
	\begin{align}
		\dot x^c \nabla_a b_c = T^b{}_{ac}\dot x^c b_b + T^b{}_{ac}\dot x^c \dot x_b \left(\frac{\Omega}{\Omega' \beta} - \frac{\beta}{\mathfrak{a}}\right)\,,
	\end{align}
	which completes the proof.
\end{proof}

In particular the choice $T^a{}_{bc}=0$ always exist, so this corollary also reproduces the result that a sufficient condition for an $(\alpha, \beta)$-Finsler geometry to be Berwald is that $b$ is covariantly constant with respect to the metric $a$. This is well-known in the positive definite case and was recently shown to hold for arbitrary signature of $g^L$ as well (Theorem 2 in \cite{Fuster:2018djw}). It also follows immediately that if $b$ is \textit{not} covariantly constant, then the connection of the resulting Finsler geometry will never be equal to the Levi-Civita connection of $a$, since $T^a{}_{bc}\neq 0$ determines the deviation of the affine connection of the Berwald space from the Christoffel symbols of $a$.

In case we choose a certain form of the tensor components $T^a{}_{bc}$ we can find the most general $(\alpha,\beta)$-Finsler Lagrangian which can be Berwald, without demanding that $\beta$ is covariantly constant.
\begin{corollary}\label{corr:3}
	Consider the most general $(1,2)$-tensor field, which is symmetric in its vector field arguments and whose components are constructed from the $1$-form $b$ and the metric $a$, not involving further derivatives of these building blocks,
	\begin{align}\label{tensorabc}
	T^a{}_{bc} = \lambda b^a b_b b_c + \rho (b_c \delta^a_b + b_b \delta^a_c) + \sigma \beta^a a_{bc}\,,
	\end{align}
	where $\lambda, \rho$ and $\sigma$ are smooth functions on $M$. In order to find $(\alpha,\beta)$-Berwald Finsler Lagrangians of the type $L = \Omega \mathfrak{a}$ with a non covariantly constant $1$-form~$\beta$, the functions must be related as
	\begin{align}
		\rho = - \sigma, \quad \mathrm{and}\quad \frac{\sigma}{\lambda} = const.\,.
	\end{align}
	Moreover, $\Omega$ must be of the form
	\begin{align}\label{eq:generalABOm}
	\Omega = \left(\frac{\beta^2}{\mathfrak{a}}\right)^{-n} \left( c + m \frac{\beta^2}{\mathfrak{a}} \right)^{n+1}\,,
	\end{align} 
	where $n$, $m$, and $c$ are constants. The $1$-form $\beta$ must satisfy
	\begin{align}\label{eq:CDgeneralABOm}
	\nabla_ab_b = q\left([c(1-n)+m g^{-1}(\beta,\beta)]b_a b_b + c n g^{-1}(\beta,\beta)g_{ab}\right)\,,
	\end{align}
	for some function $q=q(x)$.
\end{corollary}	
The proof involves lengthy derivations and is therefore displayed in Appendix \ref{app:prfthm2}. This corollary generalizes the result for m-Kropina or VGR geometries derived in \cite{Fuster:2018djw}, which is obtained by setting $c=1$ and $m=0$ in \eqref{eq:generalABOm} and \eqref{eq:CDgeneralABOm}. Moreover it is interesting to remark that up to redefinition of constants, the same condition as \eqref{eq:CDgeneralABOm} on the covariant derivative of the $1$-form was found in \cite{LiDouglas} as condition for an $(\alpha,\beta)$-Finsler geometry to be of Douglas type.

To demonstrate the application of the corollaries we discuss some explicit example geometries. On the one hand we can look for pairs of metrics and $1$-forms which satisfy \eqref{eq:CDBerwaldAB} or \eqref{eq:CDgeneralABOm}, from which we can then build $(\alpha,\beta)$-Finsler geometries:
\begin{itemize}
	\item {\bf$\mathbf{(\alpha,\beta)}$-CCNV spacetimes}: A large family of such examples arises by considering Lorentzian spacetimes with a covariantly constant null vector (CCNV), so-called CCNV spacetimes. A CCNV spacetime in $N$ dimensions is always of the (higher-dimensional) Kundt type and can be written in local coordinates $(u,v,x^1,\ldots,x^{N-2})$ as~\cite{higherCCNV}
	\begin{align}\label{CCNVMetric}
	g = 2\,\D u\D v + 2H(u,x^i)\, \D u^2 + 2W_a(u,x^i)\,\D u \D x^a +h_{ab}(u,x^i) \,\D x^a \D x^b, \nonumber\\
	a,b=1,\dots,N-2
	\end{align}
	where $u=(1/\sqrt{2})(x^{N-1}-t)$, $v=(1/\sqrt{2})(x^{N-1}+t)$ are light-cone coordinates and for real functions $H,W_a$ and a $(N-2)$-dimensional Riemannian metric $h_{ab}$, all of them independent of coordinate $v$. An interesting subclass is given by CCNV geometries with constant scalar invariants (CSI), for which $h_{ab}$ is locally homogeneous and independent of coordinate $u$. In 4D the transverse space is two-dimensional, and local homogeneity implies constant curvature. The transverse metric is then (up to isometry) that of the 2-sphere, 2D hyperbolic space or 2D Euclidean space\footnote{In this last case the resulting spacetime is in fact VSI, meaning it has \textit{vanishing} scalar invariants.} \cite{4DCSI}.
		
	For any metric of the form (\ref{CCNVMetric}), $\D u$ is a covariantly constant null $1$-form. Thus, together, they satisfy \eqref{eq:CDBerwaldAB} with $T^a{}_{bc}=0$ and therefore any $(\alpha,\beta)$-Finsler geometry constructed from them is Berwald. The functional dependence of the Finsler Lagrangian $L(\mathfrak{a},\beta)$ can be chosen nearly arbitrarily, as long as it is homogeneous and regular enough to define a Finsler geometry. In addition the Cartan non-linear connection coincides with the Levi-Civita connection of $g$, cf.\,\eqref{eq:geodsprayB}. 
	
	\item {\bf$\mathbf{(\alpha,\beta)}$-Kundt spacetimes}: Consider a $(\alpha,\beta)$-Finsler Lagrangian of the form (\ref{eq:generalABOm}), with a 4D Lorentzian metric $g$ of the Kundt type given by
	\begin{align} \label{Kundt}
	g=2du&\left(dv + H(u,v,x^i)\; du +W_1(u,x^i)\,dx^1+W_2(u,x^i)\,dx^2\right) \nonumber\\
	&+h_{ab}(u,x^i) \, \D x^a \D x^b,\;\; \;\qquad\qquad  a,b=1,2  
	\end{align}
	and the null $1$-form $du$, expressed in coordinates $(u,v,x^1,x^2)$, where $u=(1/\sqrt{2})(x^3-t)$, $v=(1/\sqrt{2})(x^3+t)$, $H$, $W_1$ and $W_2$ are real functions and $h_{ab}$ is a 2D Riemannian metric. A particularly interesting subclass is given again by those Kundt spacetimes of the form \eqref{Kundt} with constant scalar invariants, for which the transverse metric $h_{ab}$ is 2D Euclidean space and function $H$ is $H(u,v,x^i)=\Phi(u,x^i)+v\,\tilde{\Phi}(u,x^i) +v^2 \sigma$, with $\Phi$, $\tilde{\Phi}$ real functions and $\sigma$ a constant \cite{4DCSI}. If $\sigma=0$ we obtain $\mathbf{(\alpha,\beta)}$-vanishing scalar invariant (VSI) spacetimes, which generalize the Finsler VSI spacetimes presented in \cite{Fuster:2018djw}. Such a Lagrangian is of Berwald type. The 1-form is not covariantly constant and the only non-vanishing covariant derivative component is: 
	\begin{align}
	\nabla_ub_u = \tfrac{\partial H}{\partial v}\,.
	\end{align}
	It satisfies condition \eqref{eq:CDgeneralABOm}, note that $g^{-1}(\beta,\beta)=0$ in this case, with $q(x)=(\tfrac{\partial}{\partial v}H /c(n-1))$. The fact that $\beta$ is not covariantly constant implies that $T^a{}_{bc}\neq 0$, and of the form \eqref{tensorabc} with
	\begin{align}
	\sigma = \frac{n}{1-n}\tfrac{\partial H}{\partial v},\qquad \rho=-\sigma,\qquad \lambda = \frac{m}{n c}\sigma\,
	\end{align}	 
	where $c$, $m$, $n$ are constants appearing in the Finsler Lagrangian (\ref{eq:generalABOm}). 
	
	{\textit{Remark}}. The previous examples seem to indicate that the theorem found in \cite{Gomez-Lobo:2016qik}, which states that a  m-Kropina/VGR Finsler spacetime with 1-form $b=du$ and a Lorentzian metric $g$ of the Kundt type with $W_{i,v}=0$ and Euclidean transverse space is always Berwald, could generalize to Finsler Lagrangians of the $(\alpha,\beta)$-type with 1-form $b=du$ and Kundt metrics with $W_{i,v}=0$ and non-Euclidean transverse space.
\end{itemize}

On the other hand, Corollary \ref{corr:AB} can be used to find that certain Finsler Lagrangians are Berwald if and only if $\beta$ is covariantly constant: 
\begin{itemize}
	\item \textbf{Randers geometry}, $\Omega = \left(1+\sqrt{\frac{\beta^2}{\mathfrak{a}}}\right)^2$: For this case \eqref{eq:CDBerwaldAB2} becomes
	\begin{align}
		\dot x^c \nabla_a b_c = T^b{}_{ac}\dot x^c b_b + T^b{}_{ac}\dot x^c \dot x_b \frac{1}{\sqrt{\mathfrak{a}}}\,.
	\end{align}
	For any non-trivial choice of the tensor components $T^a{}_{bc}$ the term $T^b{}_{ac}\dot x^c \dot x_b$ is quadratic in $\dot x$ and can never cancel the $(\sqrt{\mathfrak{a}})^{-1}$ term. Thus it is impossible to make the right hand side of this expression linear in $\dot x$, as it is necessary for a solution to exist. Hence the only choice to have a chance to satisfy this equation is to set $T^a{}_{bc}(x) = 0$, which is a known result for positive definite Finsler geometries.
	
	\item An \textbf{exponential Finsler Lagrangian}, $\Omega =  e^{-\frac{\beta^2}{\mathfrak{a}}}$: For this case \eqref{eq:CDBerwaldAB2} becomes
	\begin{align}
	\dot x^c \nabla_a b_c = T^b{}_{ac}\dot x^c b_b - T^b{}_{ac}\dot x^c \dot x_b \left( \frac{1}{\beta} + \frac{\beta}{\mathfrak{a}}\right)\,.
	\end{align}
	Again it is clear that the last term can never be linear in $\dot x$ for any choice of $T^a{}_{bc}$ since it is impossible that $T^b{}_{ac}\dot x^c \dot x_b$ cancel the different denominators in the bracket simultaneously. Hence, as in the Randers case, these exponential kind of Finsler geometries can only be of Berwald type if $\nabla_a b_b = 0$, which is a new result.
\end{itemize}

%%%%%%%%%%%%%%%%%%%%%%%%%%%%%%%%%%%%%%%%%%%%%%%%%%%%%%%%%%%%%%%%%%%%%%%%%%%%%%%%%%%%%%%
\section{Conclusion}\label{sec:conc}
In Theorem~\ref{thm:1} we found a simple geometric partial differential equation, which classifies under which conditions a given Finsler Lagrangian is of Berwald type. Equation \eqref{eq:BerwCond} states that a Finsler Lagrangian written in the form $L = \Omega(x,\dot x) g(\dot x, \dot x)$, with $g$ being an auxiliary Riemannian metric, is of Berwald type if and only if the $0$-homogeneous factor $\Omega$ has a specific behaviour under the tangent bundle lift of the Levi-Civita covariant derivative of the  metric $g$. It needs not be covariantly constant with respect to this derivative, but the deviation of being covariantly constant must be sourced by a $(1,2)$-tensor field $T$ on the base manifold $M$ in a very specific way. This criterion identifies Berwald Finsler geometries among Finslerian geometries in a simple way. Moreover, we demonstrated that it is not necessary to choose a Riemannian metric for this decomposition, but any auxiliary (possibly indefinite) metric $\tilde g$ can be used, taking into consideration only the subbundle of $\mathcal A$ where $\tilde g(\dot x, \dot x)$ is non-vanishing, as we discussed below \eqref{eq:differentmetrics}.

The application of this result to $(\alpha,\beta)$-Finsler Lagrangians in Corollary \ref{corr:AB} leads to a necessary and sufficient condition on the Levi-Civita covariant derivative of the $1$-form defining $\beta$, which makes the geometry of Berwald type. In particular, in Corollary~\ref{corr:3}, we demonstrated how a specific choice of the tensor $T$ leads to new $(\alpha,\beta)$-Finsler Lagrangians, generalized m-Kropina/VGR geometries, which can be of Berwald type without demanding that $\beta$ is covariantly constant.

Since Berwald geometries are Finsler geometries close to (pseudo)-Riemannian geometries, their identification is of great value for the application of Finsler geometry in physics, since theories of gravity based on Berwald geometry can be seen as theories which are close to general relativity \cite{Fuster:2018djw,Hohmann:2018rpp,Pfeifer:2011xi}.

An interesting future research direction is to answer the question if the Berwald condition found here is related to the existence of a special frame basis of the tangent spaces of $T_xM$, as it was found in \cite{Gomez-Lobo:2016qik} for the case of m-Kropina/VGR Finsler Lagrangians.

\section*{Acknowledgments}
CP thanks N. Voicu, M. Hohmann and V. Perlick for many intensive and productive discussions and collaboration on Finsler geometry. CP was supported by the Estonian Ministry for Education and Science and the European Regional Development Fund through the Center of Excellence TK133 ``The Dark Side of the Universe'' and  was funded by the Deutsche Forschungsgemeinschaft (DFG, German Research Foundation) - Project Number 420243324. AF thanks S. Hervik for his feedback on CSI spacetimes and A.\,Ach\'{u}carro for everything. The work of A.\,Fuster is part of the research program of the Foundation for Fundamental Research on Matter (FOM), which is financially supported by the Netherlands \mbox{Organisation} for Scientific Research (NWO).

\appendix

%%%%%%%%%%%%%%%%%%%%%%%%%%%%%%%%%%%%%%%%%%%%%%%%%%%%%%%
\section{Proof of Corollary 3}\label{app:prfthm2}
In this appendix we display the details for the proof of corollary \ref{corr:3}.

\begin{proof}[Proof of Theorem 2]
	The most general tensor components $T^a{}_{bc}$, which can be constructed from the components of a $1$-form and a metric on $M$, $b_a$ and $g_{ab}$, is characterized by three functions $\lambda(x),\rho(x)$ and $\sigma(x)$ and is given by 
	\begin{align}
	T^a{}_{bc} = \lambda \beta^a b_b b_c + \rho (b_c \delta^a_b + b_b \delta^a_c) + \sigma \beta^a g_{bc}\,.
	\end{align} 
	Inserting this expression into the result of Corollary~\ref{corr:AB}, see equation \eqref{eq:CDBerwaldAB2}, yields
	\begin{align}\label{eq:CDbetaThm2}
	\dot x^c \nabla_a b_c 
	&= b_a \frac{\rho(\mathfrak{a}^2 \Omega + \beta^2 \mathfrak{a} \Omega') + \lambda (\beta^2 \mathfrak{a} \Omega +(\beta^2 g^{-1}(\beta,\beta)\mathfrak{a} - \beta^4)\Omega' )}{\beta \mathfrak{a} \Omega'}\nonumber\\ 
	&+ \dot x_a \left(\sigma g^{-1}(\beta,\beta) + (\sigma+\rho) \left(\frac{\Omega}{\Omega'} - \frac{\beta^2}{\mathfrak{a}}\right)\right)\,.
	\end{align}
	In order for this equation to have a solution for the $1$-form components $b_a$, both terms on the right hand side must be linear in their dependence on $\dot x$. For the second term, proportional to $\dot x_a$, this means that either $\rho=-\sigma$ or that
	\begin{align}
	\dot{\partial}_a \left( \frac{\Omega}{\Omega'} - \frac{\beta^2}{\mathfrak{a}}\right)=0\,.
	\end{align}
	Setting $s=\frac{\beta^2}{\mathfrak{a}}$ this partial differential equation can easily be solved by $\Omega = (c_1 + s) e^{c_2}$, where $c_1, c_2$, in general, can be functions on $M$. However, in order that the resulting $L$ defines a $(\alpha,\beta)$ geometry they actually must be constants. This kind of solutions are trivial since for them $L = \mathfrak{a} c + \beta^2$ is quadratic in $\dot x$ and thus not only Berwald but (pseudo)-Riemannian. 
	
	To investigate the other case assume $\rho=-\sigma$. Equation \eqref{eq:CDbetaThm2} becomes
	\begin{align}\label{eq:ACDbeta}
	\dot x^c \nabla_a b_c 
	&= -b_a \left( \beta ( \sigma  - \lambda  g^{-1}(\beta,\beta)) + \frac{\Omega}{\Omega'}\left(\sigma \frac{\mathfrak{a}}{\beta} - \lambda \beta\right) + \lambda \frac{\beta^3}{\mathfrak{a}}\right)\nonumber\\ 
	&+ \dot x_a \sigma g^{-1}(\beta,\beta)\,.
	\end{align}
	To be linear in $\dot x$, the factor multiplying $b_a$ must satisfy that its second partial derivative with respect to $\dot x$ must vanish, which is ensured if and only if
	\begin{align}
	\dot{\partial}_a\dot{\partial}_b \left(\frac{\Omega}{\Omega'}\left(\sigma \frac{\mathfrak{a}}{\beta} - \lambda \beta\right) + \lambda \frac{\beta^3}{\mathfrak{a}}\right)=0\,.
	\end{align}
	Performing the derivatives yields terms proportional to $g_{ab}$, $\dot x_a b_b$, $\dot x_b b_a$, $b_a b_b$ and $\dot x_a \dot x_b$, which all have to vanish. Their derivation is straightforward but lengthy and was done with help of the computer algebra add-on xAct for Mathematica \cite{xact}. The factor multiplying $g_{ab}$ is given by
	\begin{align}
		2 \left(\frac{\Omega}{\Omega'} \left( \frac{\sigma}{\beta} - \frac{2 \lambda \beta^3 \Omega''}{\mathfrak{a}^2\Omega'} + \frac{\sigma \beta\Omega''}{\mathfrak{a}\Omega'}\right) - \sigma \frac{\beta}{\mathfrak{a}}\right)
	\end{align}
	and must be solved for $\Omega$ such that it is identically zero. Multiplying this expression by $\beta$ and $\Omega'^2$ as well as introducing again $s=\frac{\beta^2}{\mathfrak{a}}$ results in the differential equation
	\begin{align}\label{eq:diff.eq.omega}
		\Omega' \Omega \left( \sigma + s \frac{\Omega''}{\Omega' } \left( \sigma  - 2 \lambda s \right) \right)- \sigma \Omega'^2 s = 0\,.
	\end{align}

	The solution of this equation can be obtained with help of Mathematica and is
	\begin{align}\label{eq:Omegas}
	\Omega(s) = c_2 s^{-\frac{\sigma}{c_1}}(s \lambda + c_1)^{\frac{\sigma}{c_1}+1}\,,
	\end{align}
	where $c_1 = c_1(x)$ and $c_2 = c_2(x)$ are functions on $M$. To verify that we found the desired solution, we evaluate \eqref{eq:ACDbeta} with $\Omega$ as in \eqref{eq:Omegas} and see that the equation is now consistently linear in $\dot x$. Note that
	\begin{align}
	\Omega' = c_2 s^{-(\frac{\sigma}{c_1}+1)}(s\lambda+c_1)^{\frac{\sigma}{c_1}}(\lambda s - \sigma)\,,
	\end{align}
	which implies
	\begin{align}
	\frac{\Omega}{\Omega'} = \frac{s(s\lambda + c_1)}{\lambda s - \sigma}\,.
	\end{align}
	Thus, the relevant term in \eqref{eq:ACDbeta} becomes
	\begin{align}
	\frac{\Omega}{\Omega'}\left(\sigma \frac{\mathfrak{a}}{\beta} - \lambda \beta\right) + \lambda \frac{\beta^3}{\mathfrak{a}} = \frac{s(s\lambda + c_1)}{\lambda s - \sigma} \left(\frac{\sigma}{s}-\lambda\right)\beta + \lambda s \beta = -c_1 \beta\,,
	\end{align}
	which is indeed linear in $\dot x$.		
	
	In order to define a $(\alpha,\beta)$ geometry these functions must be chosen such that $\Omega$ depends only via $s$ on the tangent space points $(x,\dot x)$, so in particular on the base manifold point $x$. This requirement immediately implies that we must choose $c_1(x) = \sigma(x)/n$ for a constant $n$, which turns $\eqref{eq:Omegas}$ into
	\begin{align}
		\Omega(s) = c_2 s^{-n}\left(s \lambda + \frac{\sigma}{n}\right)^{n+1}\,.
	\end{align}
	Choosing now $c_2^{\frac{1}{n+1}} = \frac{m}{\lambda}$, where $m$ is a constant, yields
	\begin{align}
	\Omega(s) = s^{-n}\left( m s+ \frac{m}{n} \frac{\sigma}{\lambda}\right)^{n+1}\,.
	\end{align}
	Hence to obtain a $(\alpha,\beta)$ geometry the fraction $\frac{\sigma}{\lambda}$ must be constant and we can identify the constant $c=\frac{m}{n} \frac{\sigma}{\lambda}$, which finally turns \eqref{eq:Omegas} into the form \eqref{eq:generalABOm}.
	
	To verify equation \eqref{eq:CDgeneralABOm} we employ $\lambda = \frac{m \sigma}{c n}$ in \eqref{eq:ACDbeta} to obtain
	\begin{align}
	\dot x^c \nabla_a b_c = -b_a \beta\left(  \sigma  - \frac{m \sigma}{c n}  g^{-1}(\beta,\beta) - \frac{\sigma}{n} \right) + \dot x_a \sigma g^{-1}(\beta,\beta))\,.
	\end{align}
	Applying a $\dot x$ derivative and factoring $\frac{\sigma}{nc}$ gives
	\begin{align}
	\nabla_ab_b = \frac{\sigma}{nc}\left([c(1-n)+m g^{-1}(\beta,\beta)]b_a b_b + c n g^{-1}(\beta,\beta)g_{ab}\right)\,.
	\end{align}
	Last but not least $ \frac{\sigma}{nc}$ can be replaced by an arbitrary function $q$, which corresponds to the freedom that the tensor components $ T^a{}_{bc}$ are fixed only up to multiplication with an overall function.
\end{proof}

\bibliographystyle{elsarticle-num}
\bibliography{BerwaldFinsler}

\end{document}